\documentclass{amsart}

\usepackage{graphicx}
\usepackage{verbatim}
\usepackage[margin=1.3in]{geometry}

\newtheorem{theorem}{Theorem}[section]
\newtheorem{lemma}[theorem]{Lemma}

\theoremstyle{definition}
\newtheorem{definition}[theorem]{Definition}

\newtheorem{question}{Question}

\theoremstyle{remark}

\numberwithin{equation}{section}

\usepackage[usenames,dvipsnames,svgnames,table]{xcolor}
\begin{document}

\bibliographystyle{alpha}
\title[Composite Knots with symmetric union presentations]{A Short Proof of A Theorem of Tanaka  on Composite Knots with symmetric union presentations}

\author{Feride Ceren Kose}

\address{Department of Mathematics, University of Texas at Austin}

\email{fkose@math.utexas.edu}

\begin{abstract}
We present a short proof of a theorem of Tanaka that if a composite ribbon knot admits a symmetric union presentation with one twisting region, then it has a non-trivial knot and its mirror image as connected summands.
\end{abstract}
\maketitle

\section{Introduction}
A symmetric union of $J$ is a knot diagram obtained by modifying a diagram of $J \# -J$ by inserting crossings along the axis of mirror symmetry, where $-J$ is the mirror image of $J$. The construction was first introduced by Kinoshita and Terasaka \cite{KT57} in which the crossings are inserted in a single region, then was generalized to multiple regions by Lamm \cite{La00}. A knot which admits such a diagram is called a $\textit{symmetric union}$. 

Symmetric unions are ribbon which is evident by the symmetric ribbon disk in their symmetric union diagrams. However, the converse whether all ribbon knots are symmetric unions is still unknown.

\begin{question}{\label{Q1}}
Does every ribbon knot admit a symmetric union presentation?
\end{question}

This is the case for the 21 prime ribbon knots with 10 or fewer crossings \cite{La00} and for all 2-bridge knots \cite{La06}. A search performed by Seeliger \cite{Se14} showed that out of the 137 prime ribbon knots with 11 and 12 crossings 122 knots admit symmetric union presentations. In \cite{La17}, Lamm constructs four non-symmetric ribbons that generate those for which a symmetric union presentation was not found. Besides these 15 prime knots, they also generate a family of composite ribbon knots which consist of a trefoil and a knot concordant to the mirror of the trefoil. Lamm conjectures that they are non-symmetric. This family motivates our study and we prove the following.

\begin{theorem}{\label{Theorem1.1}}(\cite{Ta19})
If $K$ is a composite ribbon knot that admits a symmetric union presentation with one twisting region, then $K = K_1 \# -K_1 \# K_2$ where $K_1$ is a non-trivial knot and $K_2$ is a symmetric union with one twisting region.
\end{theorem}

Clearly it follows that if $K$ is not the connected sum of a non-trivial knot and its mirror, then it also has a prime symmetric union summand. Theorem  \ref{Theorem1.1} gives a partial answer to Lamm's conjecture. In fact, it holds in general that the connected sum of two distinct prime knots, one of which is ribbon concordant to the mirror of the other, does not admit a symmetric union presentation with one twisting region.

Our proof relies on the characterization of composite knots in terms of the non-primeness of their double branched cover, which is a theorem of Waldhausen \cite{Wa69}, and a result of Gordon and Luecke \cite{GL96} on reducible fillings.

\section*{acknowledgement}
We are grateful to Cameron Gordon and John Luecke for their encouragement, support and many insightful discussions.

\section{Set Up}

Consider $S^3$ identified as $\mathbb{R}^3 \cup \{\infty\}$ where $\{\infty\}$ denotes the point at infinity. We assume a knot diagram is depicted in the $yz$-plane. In the presence of a symmetric union of $J$, we choose the $z$-axis to be the axis of mirror symmetry shown by the dashed line in the Figure 1. The isotopy type of a symmetric union of $J$ depends not only the knot $J$ but the location of the inserted crossings. Hence, we will refer to a diagram $D$ of $J$ when it is necessary. Let $\mathcal{S}$ denote the sphere $\{y=0\} \cup \{\infty\}$. The reflection through $\mathcal{S}$ is an orientation-reversing involution on $S^3$ and leaves the diagram $D \# -D$ unchanged. Let $\tau$ denote this involution. We call the balls bounded by $\mathcal{S}$ $\textit{the left ball}$ $B_l$ and $\textit{the right ball}$ $B_r$ according to the natural position of their projections in the $yz$-plane.

\begin{figure}[ht]
    \centering
    \includegraphics[height=3.35cm]{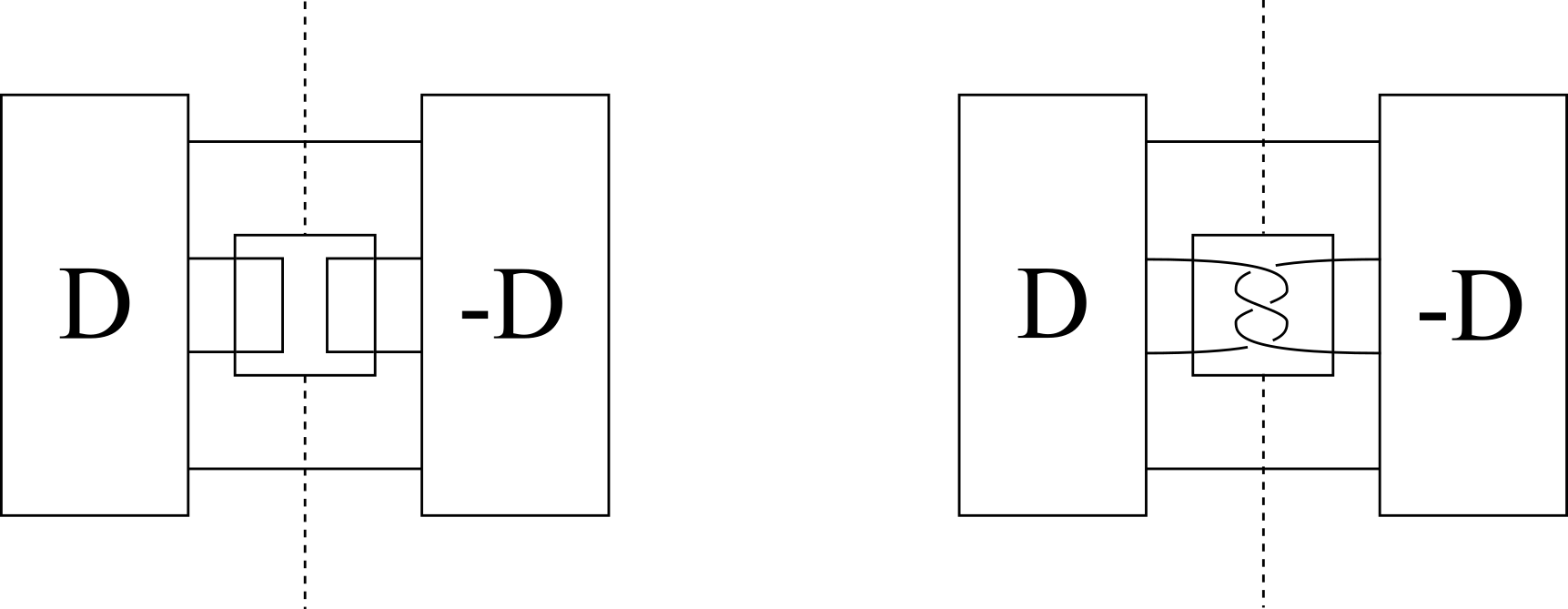}
    \caption{Schematic picture of the tangle replacement for $n=-3$}
    \end{figure}

Suppose $K$ is a symmetric union of $J$ with one twisting region. We will use the notation $(D \cup -D)(n)$ to represent $K$ where $n$ is the number of inserted crossings counted with sign. Let $B$ denote the ball containing the twisting region for the inserted crossings. $B$ meets $D \# -D$ in two trivial arcs representing a $\frac{1}{0}$-tangle, then replacing it by a $\frac{1}{n}$-tangle we obtain $K$ from $J \# -J$. See Figure 1.

$B$ lifts to a solid torus $\tilde{B}$ in the double branched cover $\Sigma_2(J\#-J)$. Then the tangle replacement corresponds to a $\frac{1}{n}$-Dehn surgery along $\tilde{B}$ giving us a description of $\Sigma_2(K)$ in terms of $\Sigma_2(J)$: $\Sigma_2(K) = \Sigma_2(J \# -J)(\frac{1}{n})$. Let $X$ denote the ball $S^3 \setminus \text{int } {B}$ and $\tilde{X}$ denote the double branched cover of $X$ branched along the tangle $T= K \cap  X$.  One can easily see that $\tilde{X}$ is $\Sigma_2(J\#-J) \setminus \text{int } \tilde{B}$ and hence, $\Sigma_2(K)$ is $\frac{1}{n}$-Dehn filling of $\tilde{X}$, $\tilde{X}(\frac{1}{n})$.

\section{Proof of Theorem}
We will prove the contrapositive of the statement of Theorem \ref{Theorem1.1}. We first show that $\tilde{X}$ is irreducible if $K$ is not of the form stated in Theorem \ref{Theorem1.1}. Having $\tilde{X}$ irreducible allows us to use the result of Gordon and Luecke on reducible fillings. We then use this result together with the structure given by the symmetry to conclude no non-trivial $\frac{1}{n}$-Dehn filling of $\tilde{X}$ yields a reducible manifold. Hence, for the remainder of the note we assume $K$ is a composite ribbon knot which is not of the form stated in Theorem \ref{Theorem1.1}.

\begin{definition}
Let $G$ be a finite group acting on a connected 3-manifold $M$. A subset $N$ of $M$ is \textit{G-equivariant} if for each element $g$ of $G$, either $g(N) \cap N = \emptyset$ or $g(N)=N$. 
\end{definition}

By default there is a $\mathbb{Z}_2$-action on $\tilde{X}$ for it is a double branched cover. This action is generated by the orientation-preserving involution $\iota$ whose fixed point set is the lift of the branching set. Since the restriction of $\tau$ on $X$ leaves $T$ invariant, it lifts to an orientation-reversing involution $\tilde{\tau}$ on $\tilde{X}$ whose fixed point set is the lift of the disk $\mathcal{D}$, where $\mathcal{D} = \mathcal{S} \cap X$.  Hence, we have $\mathbb{Z}_2 \times \mathbb{Z}_2$ acting on $\tilde{X}$. We will utilize this action to show $\tilde{X}$ is irreducible as a consequence of the equivariant sphere theorem due to \cite{MSY82} and \cite{Du85}.

\begin{theorem}(Equivariant sphere theorem)
Let $G$ be a finite group acting on a connected oriented 3-manifold $M$. If $\pi_2(M)$ is non-trivial, then there exists a $G$-equivariant essential sphere in $M$.
\end{theorem}

\begin{lemma}{\label{L3.3}}
$\tilde{X}$ is irreducible.
\end{lemma}

\begin{proof}
Assume for contradiction that $\tilde{X}$ is reducible, then there is a sphere in $\tilde{X}$ that does not bound a 3-ball. By the equivariant sphere theorem, we find an essential sphere $S$ in $\tilde{X}$ such that either $\iota(S) \cap S = \emptyset$ or $\iota(S)=S$, and either $\tilde{\tau}(S) \cap S = \emptyset$ or $\tilde{\tau}(S)=S$. 

If $\iota(S) \cap S = \emptyset$, then the spheres $S$ and $\iota(S)$ do not intersect $\tilde{T}$, and hence, they are mapped under the covering map $\rho$ to a sphere in $X$ that does not intersect $T$. This sphere bounds a 3-ball in $X$ which does not meet $T$ since $T \cup B$ is connected and which lifts to a disjoint union of 3-balls bounded by $S$ and $\iota(S)$, which contradicts the assumption that $S$ is essential in $\tilde{X}$. 

Therefore, $\iota$ maps $S$ to itself. Involutions on a sphere are known to be either the identity or conjugate to the antipodal map, or to rotation by $\pi$ around an axis, or to reflection across a plane. It cannot be the identity or conjugate to reflection across a plane, since the set of fixed points of $\iota$ is the tangle $\tilde{T}$ which is neither two dimensional nor has a component homeomorphic to $S^1$. If it was conjugate to the antipodal map, then the image of $S$ under the covering map would be an embedded $\mathbb{R}P^2$ in $X$, which is impossible. We then conclude that $\iota|_S$ is conjugate to rotation by $\pi$ around an axis with a fixed set $S^0$ which implies that $S$ intersects $\tilde{T}$ at exactly two points. 

So $\rho(S)$ is a sphere that also intersects $T$ at two points. Because $X$ is a ball, $\rho(S)$ bounds a ball $B^\prime$ in $X$ on one side. $B^\prime$ meets $T$ in a knotted arc; otherwise, it would lift to a 3-ball bounded by $S$, contradicting our assumption that $S$ is essential. Let $K^\prime$ denote the connected summand described by the arc $B^\prime \cap T$.

Suppose $\tilde{\tau}(S) \cap S = \emptyset$. Then $S$ does not intersect the lift of $\mathcal{D}$ and as a result $\rho(S)$ does not intersect $\mathcal{S}$. Thus, it is contained in either $B_l$ or $B_r$ and hence, so is $B^\prime$. By the symmetry, $\tau(\rho(S))$ is a sphere contained in the other side and the arc in which $\tau(B^\prime)$ meets $T$ describes the summand $-K^\prime$. Hence, $K$ has $K^\prime$ and $-K^\prime$ as connected summands. Replacing them by trivial arcs, we obtain a summand of $K$ that is a symmetric union with one twisting region. Hence, $K$ is of the form stated in Theorem \ref{Theorem1.1}, which contradicts the assumption on $K$.

Therefore, it must be that $\tilde{\tau}$ maps $S$ to itself. In this case,  $\tau$ also maps $\rho(S)$ to itself. Then, $\rho(S)$ intersects $\mathcal{D}$ in a circle which bounds the disk $D^\prime$, where $D^\prime = B^\prime \cap \mathcal{D}$. Since the endpoints of the arc $B^\prime \cap T$ are symmetric as $\rho(S)$ is fixed by $\tau$, $D^\prime$ and $T$ intersect non-trivially. There are only two points in $T \cap \mathcal{D}$ and $D^\prime$ cannot intersect $T$ in both of those points since they belong to distinct string components of $T$. Hence, we conclude that $D^\prime$ intersects $T$ in a single point. Then $D^\prime$ cuts $B^\prime$ into two balls $B^\prime_l$ and $B^\prime_r$, where $B^\prime_\alpha = B^\prime \cap B_\alpha$ for $\alpha=l,r$, each of which bounds a sphere that intersects $T$ at two points. Let $K^\prime_\alpha$ denote the connected summand described by the arc $B^\prime_\alpha \cap T$ for $\alpha=l,r$. Then we have $K^\prime = K^\prime_l \# K^\prime_r$. Since $\tau(B^\prime_l \cap T) = B^\prime_r \cap T$, $K^\prime_r = - K^\prime_l$. $K^\prime_l$ or $K^\prime_r$ is not trivial, otherwise; $K^\prime$ would also be trivial, which is a contradiction. Then $K^\prime$ is the connected sum of a non-trivial knot and its mirror image. The arc in which $S^3\setminus B^\prime$ intersects $K$ describes a summand that is a symmetric union with one twisting region and hence, we obtain $K$ to be of the form stated in Theorem \ref{Theorem1.1}, which is again a contradiction. 

Thus, we conclude that no essential sphere exists in $\tilde{X}$ and complete the proof. 

\end{proof}

Having $\tilde{X}$ irreducible in hand, we can now use the following result of Gordon and Luecke:
\begin{theorem}{\cite{GL96}}{\label{T3.4}}
Let $M$ be an orientable and irreducible 3-manifold with a torus boundary. If $M(\pi)$ and $M(\gamma)$ are reducible for distinct slopes $\pi$ and $\gamma$, then $\Delta(\pi,\gamma) = 1$.
\end{theorem}

Now we are ready to prove Theorem \ref{Theorem1.1}.

\begin{figure}[ht]
    \centering
    \includegraphics[width=15cm]{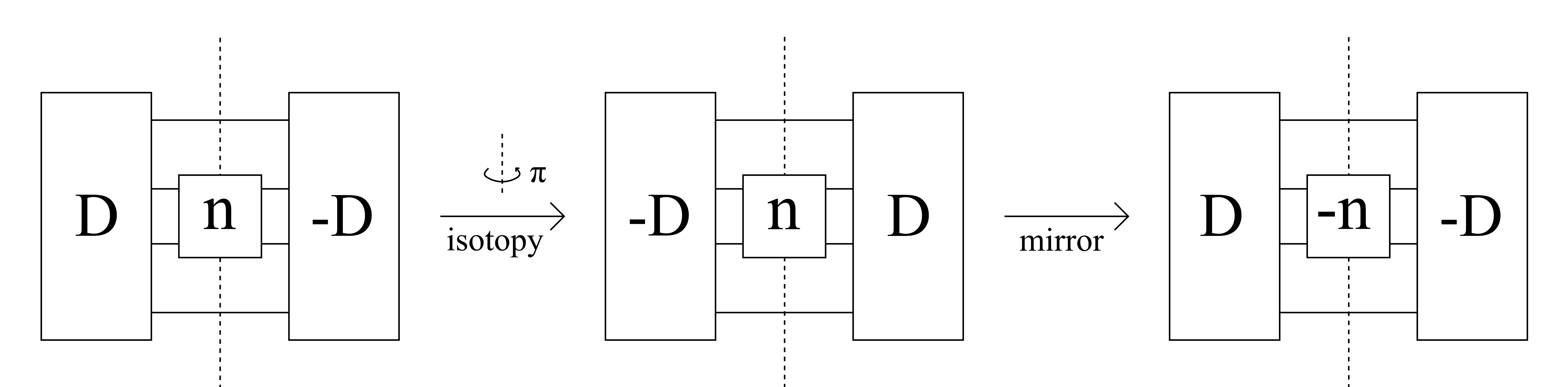}
    \caption{}
    \end{figure}

\begin{proof}[Proof of Theorem \ref{Theorem1.1}]
Assume that $K$ admits a symmetric union diagram  $(D \cup -D)(n)$ where $D$ is a diagram of a knot $J$. Then $\Sigma_2(K)$ is $\tilde{X}(\frac{1}{n})$. Since $K$ is composite, $\tilde{X}(\frac{1}{n})$ is reducible by {\cite{Wa69}}. We observe that $(D \cup -D)(-n)$ is the mirror of $(D \cup -D)(n)$, see Figure 2. Then having $\tilde{X}(\frac{1}{n})$ and $\tilde{X}(-\frac{1}{n})$ as an orientation-reserving homeomorphic pair, we know that if one of them is reducible, then the other is reducible as well. By Lemma \ref{L3.3} and Theorem \ref{T3.4} we have $1=\Delta(\frac{1}{n},-\frac{1}{n})=2|n|$, which is a contradiction.
\end{proof}

\bibliography{main}

\end{document}